\documentclass[reqno,centertags,11pt,draft]{amsart}
\usepackage[letterpaper,margin=1.3in]{geometry}
\usepackage{amsmath,amsthm,amsfonts,amssymb,enumerate}
\usepackage[bookmarksopen=true,final]{hyperref}
\usepackage{color}
\def\blue{\color{blue}}

\newcommand{\bb}{\mathbb}

\newcommand{\ca}{\mathrm{Cap}}

\renewcommand{\Re}{\mathrm{Re}}
\renewcommand{\Im}{\mathrm{Im}}

\newtheorem{theorem}{Theorem}
\newtheorem{lemma}[theorem]{Lemma}
\newtheorem{corollary}[theorem]{Corollary}
\newtheorem{proposition}[theorem]{Proposition}
\theoremstyle{definition}

\newtheorem{example}[theorem]{Example}
\theoremstyle{remark}
\newtheorem*{remark}{Remark}
\numberwithin{equation}{section}
\numberwithin{theorem}{section}

\begin{document}
\title{Widom factors in $\bb C^n$}
\author{Gökalp Alpan}
\address{Faculty of Engineering and Natural Sciences, Sabanc{\i} University, \.{I}stanbul, Turkey}
\email{gokalp.alpan@sabanciuniv.edu}

\author{Turgay Bayraktar} 
\address{Faculty of Engineering and Natural Sciences, Sabanc{\i} University, \.{I}stanbul, Turkey}
\email{tbayraktar@sabanciuniv.edu}

\author{Norm Levenberg}
\address{Mathematics Department, Indiana University, Bloomington, IN 47405 USA}
\email{nlevenbe@indiana.edu}

\maketitle
\begin{abstract} We generalize the theory of Widom factors to the $\bb C^n$ setting. We define Widom factors of compact subsets $K\subset \bb C^n$ associated with multivariate orthogonal polynomials and weighted Chebyshev polynomials. We show that on product subsets $K=K_1\times\cdots\times K_n$ of $\bb C^n$, where each $K_j$ is a non-polar compact subset of $\bb C$, these quantities have universal lower bounds which directly extend one dimensional results. Under the additional assumption that each $K_j$ is a subset of the real line, we provide improved lower bounds for Widom factors for some weight functions $w$; in particular, for the case $w\equiv 1$. Finally, we define the Mahler measure of a multivariate polynomial relative to $K\subset \bb C^n$ and obtain lower bounds for this quantity on product sets.
\end{abstract}
\section{Notation}
We study extremal polynomials in $\bb C^n$ for $n>1$ with an eye towards generalizing the theory of Widom factors to this setting. 

First, we discuss some results concerning extremal polynomials in $\mathbb C$.
Let $K$ be a non-polar compact subset of $\mathbb C$. We denote the logarithmic capacity of $K$ by $\ca(K)$ and denote the equilibrium measure of $K$ by $\mu_K$. Let $w:K\rightarrow [0,\infty)$ be a Borel measurable function such that it is positive on a subset of positive $\mu_K$-measure and $w\in L^1(K, d\mu_K)$. We consider two types of extremal polynomials. 

In the $L^2$ case, let $d\mu:=w\,d\mu_K$. Then, for each $i\in \mathbb N$, there is a unique monic polynomial $P_i$ of degree $i$ such that 
\begin{equation}
\|P_i\|_{L^2(\mu)}=\inf \{\|Q_i\|_{L^2(\mu)}\}
\end{equation}  
where the infimum is taken over all monic polynomials $Q_i$ of degree $i$. This polynomial $P_i$ is the $i$-th monic orthogonal polynomial for $\mu$.

In the sup-norm case, we further assume that $w$ is bounded above and upper semicontinuous. The weighted Chebyshev polynomial $T_{i,w}^{(K)}$ of degree $i$ is the unique monic polynomial minimizing the supremum norm $\|wQ_i\|_K=\sup_{z\in K}|w(z)Q_i(z)|$ among all monic polynomials $Q_i$ of degree $i$.

We define the $i$-th $L^2$ Widom factor for $w$ on $K$ by
\begin{equation}
W_{2,i}(K,w):= \frac{\|P_i\|_{L^2(\mu)}}{\ca(K)^i}
\end{equation} 
and 
define the $i$-th Widom factor for the sup-norm for $w$ on $K$ by
\begin{equation}
W_{\infty,i}(K,w):= \frac{\left\|T_{i,w}^{(K)}w\right\|_{K}}{\ca(K)^i}.
\end{equation}
The case when the weight function $w\equiv 1$ is of particular interest (in this case we replace $w$ with 1).

If the Erdős–Turán condition
\begin{equation}\label{w muk}
w>0\,\, \mu_K-\mbox{a.e.}
\end{equation}
holds, then (see \cite[Sec. 3]{CSZ5})

\begin{equation}
\lim_{i\rightarrow\infty}[W_{\infty,i}(K,w)]^{1/i}=1.
\end{equation}
Similarly, in the $L^2$ case, if \eqref{w muk} is satisfied then
\begin{equation}
\lim_{i\rightarrow\infty}[W_{2,i}(K,w)]^{1/i}=1
\end{equation}
in view of \cite[Theorem 4.1.1]{ST92}. Thus $d\mu=w\,d\mu_K$ is {\it regular} in the sense of Stahl-Totik if \eqref{w muk} is satisfied, i.e., the $i$-th roots of these monic orthogonal polynomials for $\mu$ exhibit regular asymptotic behavior. There is an analogous notion of regularity of measures in $\mathbb C^n$ ($n>1$), see \cite{Bloom}.

We focus on lower bounds concerning Widom factors. Let 
\begin{equation}\label{oneseven}
S(K,w):=\exp\left[\int \log{w}\,d\mu_K\right].
\end{equation}
We say that $w$ satisfies the Szeg\H{o} condition on $K$ if $S(K,w)>0$. Note that if $w$ satifies the Szeg\H{o} condition then \eqref{w muk} also holds; hence $d\mu=w\,d\mu_K$ is regular in the sense of Stahl-Totik. 
Another trivial observation is that $S(K,1)=1$. 

We have the following universal lower bounds (see \cite[Theorem 1.2]{Alp19} and \cite[Theorem 13]{NSZ21})
\begin{align} 
[W_{2,i}(K,w)]^2&\geq S(K,w),\label{wid1}\\
W_{\infty,i}(K,w)&\geq S(K,w).\label{wid2}
\end{align}
Both \eqref{wid1} and \eqref{wid2} are optimal in the sense that for arbitrary compact $K\subset \mathbb R$, and each $i\in \mathbb{N}$, we can find weight functions $w$ such that 
$$[W_{2,i}(K,w)]^2/ S(K,w) \ \hbox{and} \ W_{\infty,i}(K,w)/S(K,w)$$ are arbitarily close to $1$, see \cite[Theorem 2.2]{AlpZin21} and \cite[Theorem 13]{NSZ21}. 

Despite this, there are special scenarios where the lower bounds can be improved. One line of results concerns weights or sets for which a \emph{doubling} of the $S(K,w)$ bound holds when $K$ is a subset of the real line. That is, in certain cases one has 
\begin{equation}\label{doubled sup}
W_{\infty,i}(K,w)\ge 2\,S(K,w)
\end{equation}
or
\begin{equation}\label{doubled l2}
[W_{2,i}(K,w)]^2 \ge 2\,S(K,w),
\end{equation}
for all $i\in\mathbb{N}$ or where these inequalities hold for $i$ sufficiently large. An important particular case for which \eqref{doubled sup} and \eqref{doubled l2} hold is when $w\equiv 1$ on an arbitrary non-polar compact set $K\subset \mathbb{R}$, see \cite{Sch08} and \cite{AlpZin20}. We refer the reader to \cite{Alp22}, \cite{AlpZin21}, \cite{AZ24}, \cite{CR24}, \cite{SchZin} for further results concerning doubled lower bounds.

In addition to these non-asymptotic results on Widom factors, there are many results involving their asymptotic behavior.  

When $K$ is the unit circle $\partial \mathbb D$, the Szegő theorem states that
\[
\lim_{i\to\infty} \left[ W_{2,i}(K,w) \right]^2 = S(K,w)
\]
for weight functions $w$ on $\partial \mathbb D$ (see, e.g. \cite[Theorem 2.3.1]{Sim05}). 

For the case $K=[-1,1]$, the Szegő theorem takes the form
\begin{align*}
	\lim_{i\to\infty}\left[ W_{2,i}(K, w) \right]^2 = 2 S(K, w)
\end{align*}
for weight functions \( w \) on \( [-1, 1] \) (see, for example, \cite[Sections 13.3 and 13.8]{Sim05}). 

For a regular compact set $K\subset \mathbb R$, let $g_K$ denote the Green function of the unbounded component of the complement of $K$ and let $\{c_j\}_j$ be the set of critical points of $g_K$. Then $K$ is called a Parreau-Widom set if $\sum_j g_K(c_j)<\infty$. The class of Parreau-Widom sets includes finite gap sets and certain Cantor-type sets of positive Lebesgue measure. For such a set $K$ and a weight function $w$ on $K$, we have the following results (see \cite[eq. (1.6)]{Chr12} and \cite[Theorem 1.1]{CSZ6}):
\begin{align}
\inf_i  W_{2,i}(K,w)>0 & \implies S(K,w)>0,\\
\inf_i  W_{\infty,i}(K,w)>0 & \implies S(K,w)>0.
\end{align}

As for asymptotic lower bounds, there are several general results. If $K\subset \mathbb{R}$ is a non-polar compact set and $w$ is a weight function on $K$ then (see \cite[Theorem 4.2]{AZ24})
\begin{equation}
\liminf_{i\rightarrow\infty} [W_{2,i}(K,w)]^2\geq 2S(K,w),
\end{equation}
despite the optimality of \eqref{wid1}. Under fairly mild conditions, (see \cite[Section 3]{AZ24}, \cite[eq. (1.11)]{CSZ6})
\begin{equation}
\liminf_{i\rightarrow\infty} W_{\infty,i}(K,w)\geq 2S(K,w)
\end{equation}
holds.

When we define Widom factors (Sec. 2) and the Szeg\H{o} condition (Sec. 3) in $\mathbb{C}^n$, we make our choices in such a way that
\begin{enumerate}[$(i)$]
\item The Szeg\H{o} condition implies the regularity of the associated measure in the $L^2$ case.
\item The generalized version of $S(K,w)$ appears in the expressions concerning lower bounds of the Widom factors.
\end{enumerate}

Let $K$ be a compact set in $\mathbb{C}^n$. In order to define useful notions of Widom factors associated to $K$ when $n>1$, several complications arise. In $\mathbb{C}$, the connection with logarithmic potential theory as described, e.g., in \cite{Ran95}, is important; indeed, for a holomorphic polynomial $p(z)$ the function $u(z):=\log |p(z)|$ is subharmonic and the role of the linear Laplacian becomes clear. When $n>1$, for a holomorphic polynomial $p(z)=p(z_1,...,z_n)$, the function $u(z):=\log |p(z)|$ is now {\it plurisubharmonic}: it is uppersemicontinuous and subharmonic on complex lines. We first of all require $K$ to be a {\it non-pluripolar} compact subset of $\mathbb{C}^n$: this means that if $u$ is a plurisubharmonic function on an open, connected neighborhood of $K$ with $u|_K=-\infty$, then $u\equiv -\infty$. In particular, $K$ cannot lie in an algebraic (or even complex-analytic) subvariety of $\bb C^n$. Moreover, the natural differential operator utilized in pluripotential theory -- the study of plurisubharmonic functions -- is the highly nonlinear {\it complex Monge-Amp\`ere operator}. There is a replacement for the univariate equilibrium measure -- the normalized Monge-Amp\`ere measure $d\nu_K$ defined in (\ref{mameas}) -- and this will be utilized throughout.

Yet another complication arises as there are a multitude of possible definitions of normalized polynomial classes and hence Chebyshev and orthogonal polynomials. In particular, there are a variety of  capacitary notions, and, unlike the univariate case, the various capacities arising in pluripotential theory do not coincide on compact sets, although the reasonable ones do have the property of identifying pluripolar sets as having capacity zero. Moreover, in certain cases, Chebyshev polynomials associated to compact sets need not be unique; a simple example is given in \cite[Remark 3.4]{BBCL}.

 We briefly describe some normalized polynomial classes that will arise in definitions of Chebyshev constants and orthogonal polynomials later on. Let $e_0(z),\ldots, e_j(z),\ldots$ be a lexicographic ordering of the monomials $\{e_i(z)=z^{\alpha(i)}= z_1^{\alpha_1}\cdots z_n^{\alpha_n} \}$ in $\bb C^n$ where $\alpha(i)\in {\bb{N}_0}^{n}$, $\bb N_0=\bb N\cup \{0\}$ so that
\begin{enumerate}
	\item $\deg e_i =|\alpha(i)|=\sum_{j=1}^n \alpha_j$ is non-decreasing; and
	\item if $\alpha=(\alpha_1, \ldots, \alpha_n)\in {\bb{N}_0}^{n}$,  $\beta=(\beta_1, \ldots, \beta_n) \in {\bb{N}_0}^{n}$, and $|\alpha|=|\beta|$, we say  $\beta \prec \alpha$ if for some $l\in\{1,\ldots, n\}$, 
	\begin{equation}
		\alpha_l<\beta_l, \mbox{ and }\,\, \alpha_k=\beta_k,\,\,\, k=1,\ldots, l-1.
	\end{equation}
\end{enumerate}
Define 
\begin{equation}
	\mathcal P_i= \mathcal P(\alpha(i)):=\left\{e_i(z)+\sum_{j<i} c_j e_j(z): c_j\in \bb C \right\}.
\end{equation}
These are the ``monic'' elements in the finite dimensional vector space of polynomials  
\begin{equation}
	\mathbb P_i= \mathbb P(\alpha_i):=\left\{\sum_{j\leq i} c_j e_j(z): c_j\in \bb C \right\}.
\end{equation}
We will consider several different extremal problems related to these classes.

In the remaining subsections of this introductory section, we define (weighted and unweighted) Chebyshev polynomials and orthogonal polynomials associate to the $\mathcal P_i$ classes, and we introduce and relate various capacitary notions. Motivated by (\ref{ti}) in Theorem \ref{tau} in Section 2, we define our Widom factors in $\bb C^n$ in (\ref{wi1}) and (\ref{wi2}). In Section 3, we prove a general lower bound (\ref{jen}) in Theorem \ref{threeone} on $L^2-$norms of polynomials in terms of an exponential relative entropy function (\ref{relent}) analogous to the univariate case and use this to provide good estimates in the case of the unit Euclidean ball and unit polydisk, two prototypical compact sets in $\bb C^n$. Section 4 studies Chebyshev and orthogonal polynomials on product sets; in particular, \eqref{unif l2}, \eqref{hat wid}, Theorems 
\ref{fourthree} and \ref{foursix} generalize univariate results. Finally, in Section 5, we define the Mahler measure of a polynomial $P$ relative to $K$ as $ M(P):= \exp \left[\int \log{|P|} \, d\nu_K\right]$. We generalize a univariate result of Mahler (which estimates coefficients of polynomials in terms of $M(K)$ and the logarithmic capacity of $K$ for $K\subset \bb C$) to the multivariate setting for product sets. The Mahler measure has number-theoretic applications (cf., \cite{igor}), and we give a simple application of our results to polynomials with integer coefficients in Theorem \ref{fivethree}.

We hope that this work can provide a basis for further study of higher dimensional Widom factors.

\subsection{Weighted Chebyshev polynomials} Let $w:K\rightarrow \bb R^+$ be a nonnegative Borel measurable function such that $\sup_{z\in K} w(z)<\infty$ and the set $\{z: w(z)> 0\}$ is non-pluripolar. Let

\begin{equation}
	\widehat{w}(z):= \lim_{r\rightarrow 0^+} \sup_{\xi\in B_r(z)\cap K}w_j(\xi),\,\, z\in K,
\end{equation}
where $B_r(z)$ is the open ball of radius $r$ centered at $z$.
Hence $\widehat{w}$ is upper semicontinuous, bounded, non-negative on $K$ with 
\begin{equation}
	w(z)\leq \widehat{w}(z),\,\,\, z\in K.
\end{equation}
Let $\|\cdot\|_K$ denote the sup-norm on $K$. For any continuous $f: K\rightarrow \bb C$, we have 
\begin{equation}\label{upper eqq}
	\|f w\|_{K}= \|f \widehat{w}\|_{K}.
\end{equation}
This follows from a simple generalization of \cite[Lemma 1]{NSZ21}. For such $f$, let $\|f\|_{\widehat{w},K}:= \|\widehat{w}f\|_K$. Since $\{z: \widehat{w}(z)> 0\}$ is non-pluripolar, for $f \in \mathbb P_i$, the equation $\|f\|_{\widehat{w},K}=0$ holds if and only if $f$ is the zero polynomial as the zero set of a non-zero polynomial is pluripolar. Thus, it is easy to verify that $\|\cdot\|_{\widehat{w},K}$ is a norm on $\mathbb P_i$. Since $\mathbb P_i$ is a complex finite dimensional normed vector space, it is complete so there must be at least one solution for the following minimization problem (for $i\geq 1$):
\begin{equation}\label{cheb min}
	\inf\{\|P\|_{\widehat{w},K}: P\in \mathcal P_i\}= \inf\{\|e_i-Q\|_{\widehat{w},K}: Q\in \mathbb P_{i-1}\}.
\end{equation}
A minimizing polynomial in $\mathcal P_i$ is called an {\it $i$-th weighted Chebyshev polynomial on $K$ with respect to the weight $w$}.
Although the solution to this problem might not be unique, by abusing notation we denote a weighted Chebyshev polynomial with leading term $z^{\alpha(i)}$ for the weight $w$ on $K$ by $T_{\alpha(i),w}^{(K)}$ and we let $t_{\alpha(i),w}^{(K)}:=\|T_{\alpha(i),w}^{(K)}\|_{w,K}$. Then, in view of \eqref{upper eqq}, we get
\begin{equation}\label{one8}
	\inf\{\|P w\|_{K}: P\in \mathcal P_i\}= 	\inf\{\|P\|_{\widehat{w},K}: P\in \mathcal P_i\}.
\end{equation}
In particular, to study such minimization problems, it is enough to consider upper semicontinuous weight functions.

\subsection{Orthogonal polynomials.} For $K\subset \bb C^n$ compact and non-pluripolar, let $\mu$ be a positive finite Borel measure on $K$ such that $\mathrm{supp}(\mu)$ is non-pluripolar. Then the monomials $e_0, e_2, \ldots, e_j, \ldots$ are linearly independent in $L^2(\mu)$, see \cite[Proposition 3.5.]{Bloom}. Applying Gram-Schmidt on $\{e_i\}_{i=0}^\infty$, we obtain monic orthogonal polynomials $V_{\alpha(i)}\in \mathcal P_i$ for each $i$ where $V_{\alpha(0)}\equiv 1$. Thus
\begin{equation}
	\|V_{\alpha(i)}\|_{L^2(\mu)}= \inf \left\{\|P\|_{L^2(\mu)}: P\in \mathcal P_i \right\}.
\end{equation}
\subsection{Root asymptotics for unweighted Chebyshev polynomials.} We mention some results from \cite{Zak1}. Let 
\begin{equation}\label{Sig}
	\Sigma=\left\{\theta: (\theta_s)\in \mathbb{R}^n, |\theta|= \sum_{s=1}^n\theta_s=1,\ \theta_s\geq 0 \right\},
\end{equation}
and $\Sigma_0=\left\{\theta \in \Sigma: \theta_s > 0, \ s=1,...,n \right\}$. For $i\geq 1$, let

\begin{equation}
\tau_i = {\left[t_{\alpha(i),1}^{(K)}\right]}^{1/{|\alpha(i)|}}
\end{equation}
and
\begin{align}
\tau^+ (K)= \limsup_{j\rightarrow\infty}\tau_j,\,\, \tau^- (K)= \liminf_{j\rightarrow\infty}\tau_j.
\end{align}
For any $\theta \in \Sigma$, let
\begin{equation}
	\tau(K, \theta ):= \limsup_{\substack{i\rightarrow \infty \\\frac{\alpha(i)}{|\alpha(i)|}  \rightarrow \theta}} \tau_i,
\end{equation}
and
\begin{equation}
	\tau_{-}(K, \theta ):= \liminf_{\substack{i\rightarrow \infty \\\frac{\alpha(i)}{|\alpha(i)|}  \rightarrow \theta}} \tau_i.
\end{equation}
Here, for $\theta\in \Sigma_0$, from \cite[Lemma 1]{Zak1} we have 
\begin{equation}
	\tau(K, \theta )= \lim_{\substack{i\rightarrow \infty \\\frac{\alpha(i)}{|\alpha(i)|}  \rightarrow \theta}} \tau_i.
\end{equation}
Also, from  \cite[Lemma 3]{Zak1} for any $\theta\in \Sigma\setminus \Sigma_0$ we have
\begin{equation}\label{taulow}
	\tau_{-}(K, \theta )= \liminf_{\substack{\theta^\prime\rightarrow \theta \\ \theta^\prime \in \Sigma_0}} \tau(K,\theta^\prime).
\end{equation}
The {\it directional Chebyshev constants} $\tau(K, \theta )$ were introduced by Zaharjuta in \cite{Zak1} in order to solve a problem of Leja involving 
the notion of transfinite diameter of a compact set $K\subset \mathbb{C}^n$ when $n>1$.

For $z\in \mathbb{C}^n$, let 
\begin{align}\label{norm1}
|z|= \left(\sum_{\nu=1}^n |z_\nu|^2 \right)^{\frac{1}{2}},
\end{align}

\begin{align}\label{norm2}
\|z\|= \max\{|z_\nu|\}, \nu=1,\ldots, n.
\end{align}
Since $K$ is compact, 
$\left\|T_{\alpha(i),1}^{(K)}\right\|_K\leq \|e_i\|_K \leq (\sup_{z\in K} \|z\|)^{|\alpha(i)|}$ holds for each $i$ and thus
\begin{equation}
	\tau_i\leq \sup_{z\in K} \|z\|
\end{equation}
which implies that $\tau^+(K)$ is bounded above.

It follows from \cite[Corollary 4, Corollary 6, property 7(d') in Section 7]{Zak1} that $\tau^-(K)>0$ provided that $K$ is non-pluripolar. 

Finally, \cite[Corollary 3]{Zak1} says that 
\begin{equation}
	\tau^+(K)= \sup_{\theta\in \Sigma} \tau(K, \theta)
\end{equation}
and 
\begin{equation}\label{tau-}
	\tau^-(K)= \inf_{\theta\in \Sigma} \tau(K, \theta)=\inf_{\theta\in \Sigma_0} \tau(K, \theta).
\end{equation}

We mention that weighted versions of these notions have been studied before, cf., \cite{BL} but the setup for the latter and for our case differ since \cite{BL} considers minimizers for the problem $\|P w^{|\alpha(i)|}\|_K$, $P\in \mathcal P_i$ whereas in (\ref{one8}) we do not consider varying powers of $w$ depending on the degree. 

\subsection{Relations between capacitary notions}

Let $L_s= \inf\{\|P\|_K\}$ where the infimum is taken over the class of polynomials 
$$P(z)=\sum_{|k|\leq s} a_k z^k, \ \ \  \sum_{|k|=s} |a_k|\geq 1.$$

Then $T(K) = \lim_{s\rightarrow\infty} |L_s|^{1/s}$ exists, see p. 359 in \cite{Zak1}. It follows from definition of $L_s$ that $T(K)\leq \tau_{-}(K,\theta)$ for all $\theta$ hence
\begin{equation}\label{ttau}
	T(K)\leq \tau^{-}(K).
\end{equation}

For $K \subset \bb C^n$ compact, we let $V_K^*(z):=\limsup_{\zeta \to z}V_K(\zeta)$ where 
$$V_K(z):=\sup\{u(z):u\in L(\bb C^d) \ \hbox{and} \  u\leq 0 \ \hbox{on} \ K\}$$
$$= \sup \{{1\over {\rm deg} (p)}\log |p(z)|:p\in \bigcup {\mathcal P}_n, \  ||p||_{K}\leq 1 \}.$$
Here $u\in L(\bb C^n)$ if $u$ is plurisubharmonic in $\bb C^n$ (we write $u\in PSH(\bb C^n)$) and $u(z)\leq \log |z| +\mathcal{O}(1)$ as $|z|\to \infty$. It is known that $K$ is pluripolar if and only if $V_K^*\equiv +\infty$; otherwise $V_K^*\in L(\bb C^n)$. For non-pluripolar $K$, 
we denote by
\begin{equation} \label{mameas} d\nu_K:={\left(\frac{1}{2\pi}\right)}^n(dd^c V_K^*)^n \end{equation}
the normalized {\it Monge--Amp\`ere measure} of the extremal function $V_K^*$. Here, $d=\partial +\bar \partial$ and $d^c =i(\bar \partial - \partial)$ as in \cite{K}, p. 14. Moreover, the nonlinear {\it complex Monge--Amp\`ere operator} $(dd^c \cdot)^n$ is well defined on locally bounded plurisubharmonic functions as a positive measure. Dividing $(dd^c V_K^*)^n$ by $(2\pi)^n$ makes  $d\nu_K$ a probability measure. A compact set $K\subset \bb C^n$ is {\it regular at $a\in K$} if $V_K$ is continuous at $a$; $K$ is {\it regular} if it is continuous on $K$, i.e., if $V_K=V_K^*$. Let
$$c(K)= \exp{\left[-\limsup_{z\rightarrow\infty}(V_K^*(z)- \log \|z\|)\right] }$$
and
$$C(K)= \exp{\left[-\limsup_{z\rightarrow\infty}(V_K^*(z)- \log |z|)\right] }.$$
Clearly, we have $c(K) \leq C(K)$. It was shown in \cite[Theorem 2]{Zak1} that $c(K)\leq T(K)$. 

It was claimed in \cite[Thereom 25.2]{Sad} that $C(K)\leq T(K)$. The proof given is a repetition of the proof of \cite[Theorem 2]{Zak1}. According to Zakharyuta (p. 300 in \cite{Zak2}) the true inequality is $c(K)\leq T(K)$. Indeed, in Example \ref{oneone} below we give an example where $T(K)<C(K)$. Sadullaev proves the converse inequality $T(K)\leq C(K)$ in \cite[Proposition 25.3]{Sad} for regular sets but this can be easily extended to general compact sets by an approximation argument  as for any compact $K$ and $\epsilon >0$, the set $K_{\epsilon}:=\{z: \hbox{dist}(z,K)\leq \epsilon\}$ is regular; cf., \cite{K}, Corollary 5.1.5. Thus, we have 
\begin{equation}
	c(K)\leq T(K)\leq C(K)
\end{equation}
for arbitrary compact sets as noted by Zakharyuta in p. 300 of  \cite{Zak2}. In addition, Sadullaev's proof of $T(K)\leq C(K)$ is also applicable for $\tau^{-}(K)$ instead of $T(K)$ so that we also have $\tau^{-}(K)\leq C(K)$. Thus

\begin{equation}\label{various cap}
	c(K)\leq T(K)\leq \tau^{-}(K) \leq C(K).
\end{equation}

Let $K=B_2=\{(x_1,x_2)\in \bb R^2: x_1^2 +x_2^2\leq 1\}$ be the closed real Euclidean unit ball in $\bb R^2$ embedded into  $\bb C^2$. We discuss several properties of this set below.
\begin{example} \label{oneone} Orthogonal and Chebyshev polynomials on $K=B_2$ were studied in \cite{BBL} and we derive some consequences using that paper. By \cite[Lemma 2.4]{BBL}, for $\theta=(\theta_1,\theta_2)\in \Sigma_0$, we have

\begin{equation}\label{the b2}
[\tau(K, \theta)]^2=[\tau(K, (\theta_1,1-\theta_1))]^2=\frac{{\theta_1}^{\theta_1} {(2- {\theta_1})}^{2- {\theta_1}}}{4^{2-\theta_1}}
\end{equation}
If we take the limit in \eqref{the b2} as $\theta_1\rightarrow 0$ or as $\theta_1\rightarrow 1$, and use \eqref{taulow}, we see that
\begin{equation}\label{b2 boundary}
	\tau_{-}(K,(1,0))= \tau_{-}(K,(0,1))=\frac{1}{2}.
\end{equation}
By taking the logarithm of the right hand side  of \eqref{the b2} and then differentiating, we notice that the only critical point of that function for $\theta_1\in (0,1)$ occurs at $\theta_1=\frac{2}{5}$. Upon substitution, we see that
\begin{equation}\label{b2 min}
\tau(K, (2/5, 3/5))=\frac{2}{5}.
\end{equation}\label{b2 tau minus}
Combining \eqref{b2 boundary} and \eqref{b2 min}, we get
\begin{equation}
\tau^{-}(K)=\frac{2}{5}.
\end{equation}
Next, we compute $c(K)$ and $C(K)$. 
By Lundin's formula \cite{lundin} , 
\begin{align}
V_K^{*}(z)&=V_K^{*}(z_1,z_2)\\
&= \frac{1}{2}\log[|z_1|^2+|z_2|^2+ |z_1^2+z_2^2-1|+\sqrt{(|z_1|^2+|z_2|^2+ |z_1^2+z_2^2-1|)^2-1}
\end{align}
for $z=(z_1,z_2)\in \bb C^2\setminus K$.
Since 
\begin{equation}
|z_1^2+z_2^2-1|\leq |z_1|^2+|z_2|^2+1= |z|^2+1,
\end{equation}
it follows that
\begin{equation}\label{green upper}
V_K^{*}(z)\leq \frac{1}{2}\log{(4|z|^2+2)}
\end{equation}
Thus,
\begin{equation}\label{along some path}
\limsup_{z\rightarrow\infty}(V_K^*(z)- \log |z|)\leq \frac{1}{2}\log{4}= \log{2}
\end{equation}
Now if $\Re{z_1}\rightarrow +\infty$, $\Im{z_1}=z_2=0$, then along that path $V_K^*(z)- \log |z|\to \log{2}$ which combined with \eqref{along some path} imply that
\begin{equation}\label{general path}
\limsup_{z\rightarrow\infty}(V_K^*(z)- \log |z|)= \log{2}.
\end{equation}
Thus
\begin{equation} 
C(K)=e^{-\limsup_{z\rightarrow\infty}(V_K^*(z)- \log |z|)}=\frac{1}{2}
\end{equation}
and hence $T(K)\leq \tau^{-}(K)<C(K).$

Note that in $\bb C^2$, the norms (\ref{norm1}) and (\ref{norm2}) satisfy
\begin{equation}
|z|^2\leq 2||z||^2.
\end{equation}
Thus, in view of \eqref{green upper}, we get
\begin{equation}
\limsup_{z\rightarrow\infty}(V_K^*(z)- \log \|z\|)\leq \frac{1}{2}\log{8}= \log{2\sqrt{2}}.
\end{equation}
Along the path $\Re{z_1}=\Re{z_2}$ and $\Im{z_1}=\Im{z_2}=0$, $V_K^*(z)- \log \|z\| \to \log{2\sqrt{2}}$. Therefore,
\begin{equation}
c(K)= \exp{\left[-\limsup_{z\rightarrow\infty}(V_K^*(z)- \log \|z\|)\right] }=\frac{1}{2\sqrt 2}.
\end{equation}
Thus $c(K)<\tau^{-}(K)<C(K).$

\end{example}
\textbf{Acknowledgments}: Research of Gökalp Alpan was supported by the Scientific and Technological Research Council of Türkiye (TÜBİTAK) ARDEB 1001 Grant Number 123F358. Research of Gökalp Alpan was also supported by BAGEP Award of the Science Academy. 

T. Bayraktar is partially supported by TÜBİTAK-ARDEB grant no 123F358.

Norm Levenberg is supported by Simons Foundation grant No. 707450.

\section{Widom factors in $\bb C^n$} 
Our objective is to generalize the theory of Widom factors to the $\bb C^n$ setting. Let $\ca(\cdot)$ denote the logarithmic capacity in $\bb C$ as in the introduction. We first obtain a simple result involving $\tau^{-}(K)$ for $K\subset \bb C^n$ instead of $\ca(K)$ for $K\subset \bb C$. This can be considered a generalization of \cite[Theorem 5.5.4(a)]{Ran95}. We have the following lower bound for Chebyshev polynomials:

\begin{theorem}\label{tau}
	Let $K$ be a non-pluripolar compact subset of $\mathbb{C}^n$. Then for any $i$, we have
	\begin{equation}\label{ti}
		\left\|T_{\alpha(i),1}^{(K)}\right\|_K \geq  [\tau^{-}(K)]^{|\alpha(i)|}.
	\end{equation}
\end{theorem}
\begin{proof}
For $i=0$, this is trivial.	Fix $i\geq 1$. Define $Q_N= \left[T_{\alpha(i),1}^{(K)}\right]^N$ for each $N\geq 1$. Then $Q_N$ is a polynomial whose leading term is $z^{N\alpha(i)}$. Note that the value of $\|Q_N\|_K^{1/{(N|\alpha(i)|)}}$ is independent of $N$. Thus, it follows from the definition of $\tau^-(K)$ that
\begin{equation}\label{exp}
\|Q_1\|_K^{1/{(|\alpha(i)|)}}=	\lim_{N\rightarrow\infty} \|Q_N\|_K^{1/{(N|\alpha(i)|)}}=	\liminf_{N\rightarrow\infty} \|Q_N\|_K^{1/{(N|\alpha(i)|)}}\geq \tau^{-}(K).
\end{equation}
Exponentiating the terms in \eqref{exp} by $|\alpha(i)|$ yields \eqref{ti}.
\end{proof}
As a corollary, if we combine \eqref{various cap} and Theorem \ref{tau}, we get
\begin{equation}\label{yet another ineq}
\left\|T_{\alpha(i),1}^{(K)}\right\|_K\geq [\tau^{-}(K)]^{|\alpha(i)|} \geq [T(K)]^{|\alpha(i)|}\geq c(K)^{|\alpha(i)|}.
\end{equation}

Theorem \ref{tau} helps motivate our definition of Widom factors. 
\medskip

{\bf Definition of the $L^2$ Widom factors}: Let $K$ be a non-pluripolar compact subset of $\bb C^n$. Let $\nu_K$ denote the Monge--Amp\`ere measure of $K$ from (\ref{mameas}). Let $w$ be a Borel measurable non-negative function $w:K\rightarrow [0,\infty]$ such that $w\in L^1(\nu_K)\setminus\{0\}$ and so $\mathrm{supp}(w\, d\nu_K)$ is non-pluripolar (see e.g. \cite[Theorem A.III.2.5]{ST97}) . Let $V_{\alpha(i)}\in \mathcal P(\alpha(i))$ be an orthogonal polynomial for the measure $w\, d\nu_K$ with leading term $z^{\alpha(i)}$. We define the $\alpha(i)$-th $L^2$ Widom factor for $w$ on $K$ by
\begin{equation}\label{wi1}
	W_{2,\alpha(i)}(K,w):= \frac{\|V_{\alpha(i)}\|_{L^2(w\,d\nu)}}{[\tau^{-}(K)]^{|\alpha(i)|}}.
\end{equation}

{\bf Definition of the Widom factors for the sup norm}: Let $K$ be a non-pluripolar compact subset of $\bb C^n$ and $\widehat{w}:K\rightarrow [0,\infty)$ be a bounded upper semicontinuous weight function such that $\{z: \widehat{w}(z) > 0 \}$ is non-pluripolar. We define the  $\alpha(i)$-th Widom factor for the sup-norm for the weight function $\widehat{w}$ on $K$ by
\begin{equation}\label{wi2}
	W_{\infty,\alpha(i)}(K,\widehat{w})= \frac{\left\|T_{\alpha(i),\widehat{w}}^{(K)}\right\|_{K}}{[\tau^{-}(K)]^{|\alpha(i)|}}.
\end{equation}

Note that, \eqref{yet another ineq} immediately implies that
\begin{equation}
W_{\infty,\alpha(i)}(K,1)\geq 1, \,\,\, \mbox{ for } i\geq 1.
\end{equation}

In the sequel, we will justify these definitions by finding additional generalizations of one dimensional results concerning Widom factors. 
We next give explicit calculations for the unit polydisk and the Euclidean unit ball.

\begin{example}\label{triv ex}
	 Let $K=\{z\in \bb C^n: ||z||\leq 1\}$ be the unit polydisk. The unweighted Chebyshev polynomials are the monomials $z^{\alpha(i)}$. Here, $\left\|z^{\alpha(i)}\right\|_K=1$ for all indices. Therefore $\tau^{-}(K)=1$. This implies that we have the equality 
\begin{equation}
W_{\infty,\alpha(i)}(K,1)= 1.
\end{equation}
	for each $i$. This is a generalization of the well-known result for $\overline{\mathbb{D}}$ on $\mathbb{C}$.
\end{example}

\begin{example}\label{ex2}
Let $K=\{(z=(z_1,z_2): |z_1|^2+|z_2|^2\leq 1\}$ be the Euclidean unit ball in $\bb C^2$. For any $\alpha(i)=(\alpha_1, \alpha_2)\in \bb N_0^2$, the monomial $z^{\alpha(i)}$ is a Chebyshev polynomial for $w=1$ (see e.g. \cite[Proposition 4]{BloCal}). It is easy to see that the sup-norm of $z^{\alpha(i)}$ on $K$ occurs on the set $|z_1|^2+|z_2|^2=1$. Thus the problem to compute $\left\|T_{\alpha(i),1}^{(K)}\right\|_K$ reduces to computing the maximum value of $|z_1|^{\alpha_1} \left(\sqrt{1-|z_1|^2}\right)^{\alpha_2}$ on $|z_1|\leq 1$. It is easy to see that the maximum is realized if $|z_1|^2=\frac{\alpha_1}{\alpha_1+\alpha_2}$. With the convention $0^0:=1$, we have 
\begin{equation}\label{alpha version}
  \left\|z^{\alpha(i)}\right\|_K= \frac{\alpha_1^{\alpha_1/2}\alpha_2^{{\alpha_2}/2}} {(\alpha_1+\alpha_2)^{(\alpha_1+\alpha_2)/2}}.
\end{equation}  

For any $\theta=(\theta_1, \theta_2)\in \Sigma_0$ we have
\begin{equation}\label{euc ball}
	\tau(K,\theta)= \theta_1^{\theta_1/2}\theta_2^{\theta_2/2}.
\end{equation}
It is not difficult to see that $\tau(K,\theta)$ is minimized if $\theta_1=\theta_2=1/2.$ Hence 
\begin{equation}
	\tau^{-}(K)= 1/\sqrt{2}.
\end{equation}
For $\alpha(i)=(N,N)$,
$$\left\|z^{\alpha(i)}\right\|_K=\frac{1}{2^N}.$$
This implies that, for such $\alpha(i)$, 
\begin{equation}
W_{\infty,\alpha(i)}(K,1)= 1.
\end{equation}
For the directions $(1,0)$ and $(0,1)$, in view of \eqref{alpha version}, \eqref{euc ball} and \eqref{taulow}, we have
\begin{equation}\label{tau counter}
	\tau_{-}(K,(1,0))= \tau(K,(1,0))=\tau_{-}(K,(0,1))= \tau(K,(0,1))=1.
\end{equation} 
Next, we calculate $c(K)$ and $C(K)$. For $z\in \bb C^2\setminus K$, (see \cite[p. 32]{Lev12}) 
\begin{equation}
V_K(z)=V_K^*(z)=\log^+|z|.
\end{equation}
Thus, it is easy to see that $c(K)=\frac{1}{\sqrt 2}$ and $C(K)=1$. Hence for this case, we have 
$$c(K)=T(K)=\tau^{-}(K)<C(K).$$
\end{example}

\begin{example}
We can get an upper bound on $\tau^{-}(K)$ for a large class of compact sets $K$ by using the inequality 
$\tau^{-}(K) \leq C(K)$ from (\ref{various cap}) and bounding $C(K)$ from above. To describe this class, we 
begin by noting for the standard simplex 
$$S:=\{x=(x_1,...,x_n)\in \mathbb R^n: x_j\geq 0, \ \sum_{j=1}^n x_j \leq 1\}$$
in $\mathbb R^n$ we have
$$V_S(z)=\log h(|z_1|+\cdots +|z_n| +|z_1 +\cdots + z_n -1|)$$
where $h(t)=t+\sqrt{t^2-1}$ (cf., \cite{K}, Example 5.4.7). Then
$$\limsup_{|z|\to \infty}[V_S(z)-\log |z|]=\log (4\sqrt n)$$
where $|z|^2:=\sum_{j=1}^n |z_j|^2$. The $\sqrt n$ comes from the upper bound in 
the $l^1-l^2$ norm comparison:
$$|z_1|+\cdots +|z_n|\leq \sqrt n \cdot [\sum_{j=1}^n |z_j|^2]^{1/2}$$
with equality if, e.g., $|z_1|=\cdots =|z_n|$. It follows that $$C(S)=\frac{1}{4\sqrt n}.$$

Now let $K$ be a general $n-$simplex in $\mathbb R^n$. We can write 
$K=T(S)$ where $T$ is a real affine map; i.e., 
$$T(x)= a + Ax, \ x\in \mathbb R^n,$$
where $a\in \mathbb R^n$ and $A\in GL(n,\mathbb R)$. In particular, $T(z)$ is a complex affine map 
in $\mathbb C^n$ and from \cite{K}, Sec. 5.3,
$$V_K(z)=V_S(T^{-1}(z)).$$
Then
$$\limsup_{|z|\to \infty}[V_K(z)-\log |z|]$$
$$=\limsup_{|z|\to \infty}[V_S(T^{-1}(z))-\log |T^{-1}(z)|+\log |T^{-1}(z)|-\log |z|]$$
Now observe that the computation of 
$$\limsup_{|z|\to \infty}[V_S(T^{-1}(z))-\log |T^{-1}(z)|]$$
is the same as that of 
$$\limsup_{|z|\to \infty}[V_S(z)-\log |z|]$$
for we encounter the same $l^1-l^2$ norm comparison now for $T^{-1}(z)$ instead of $z$. For an 
upper bound on 
$$\log |T^{-1}(z)|-\log |z|] =\log \frac{|T^{-1}(z)|}{|z|},$$
we can take the log of the absolute value of the biggest eigenvalue $\Lambda$ of $(A^{-1})^T A^{-1}$. With this notation,
$$\limsup_{|z|\to \infty}[V_K(z)-\log |z|]\leq \log (4\sqrt n)+ \log |\Lambda|$$
and hence
$$C(K)\geq C(S) \frac{1}{|\Lambda|}.$$


Now let $K$ be an arbitrary compact, convex polytope in $\mathbb R^n$. Then $K$ can be obtained as the intersection of a finite collection of simplices or strips $S(K)$. S. Ma'u \cite{sione} has proved the following formula for the extremal function $V_K$ of such a set:
$$V_K(z)=\max\{V_E(z): E\in S(K)\}.$$
Thus we have 
$$\limsup_{|z|\to \infty}[V_K(z)-\log |z|]\geq \max_{E\in S(K)} \bigl( \limsup_{|z|\to \infty}[V_E(z)-\log |z|]\bigr)$$
so that
$$C(K)\leq \min_{E\in S(K)} C(E),$$
giving a calculable upper bound on $\tau^{-}(K)$.
\end{example}

\section{General Lower bounds}
We first start with a general result.
Let $K$ be a non-pluripolar compact subset of $\mathbb{C}^n$. For a Borel measurable weight function $w:K\rightarrow [0,\infty]$ with $w\in L^1(\nu_K)\setminus \{0\}$, in analogy with (\ref{oneseven}) we define the exponential relative entropy function by
\begin{equation}\label{relent}
	S(K,w)=\exp\left[\int \log{w}\,d\nu_K\right].
\end{equation}  
We say that $w$ satisfies the Szeg\H{o} condition on $K$ if $S(K,w)>0$. We set $e^{-\infty}:=0$.

Note that $S(K,w)>0$ implies $\mathrm{supp}(w d\nu_K)=\mathrm{supp}(d\nu_K)$. Since the support of $\nu_K$ is non-pluripolar (see e.g. \cite[Appendix A.III.2.5]{ST97}), this also implies $\mathrm{supp}(w d\nu_K)$ is non-pluripolar in that case. Thus the Szeg\H{o} condition ensures the linear independence of the monomials in $L^2(wd\nu_K)$. In addition, $S(K,w)>0$ implies that $w>0$ $\nu_K$-a.e. and in particular regularity of $d\mu=w\,d\nu_K$, (see \cite[Proposition 3.5.]{LevFrank24} and \cite[p. 442]{Bloom}). 

The case when $S(K,w)=0$ gives only trivial lower bounds for the norms of extremal polynomials. Hence we always assume the Szeg\H{o} condition for the statements of the results below.

\begin{theorem}\label{threeone} 
	Let $K$ be a non-pluripolar compact subset of $\bb C^n$ and $d\mu=w\, d\nu_K$ a finite positive Borel measure with $S(K,w)>0$. Let $P$ be a polynomial. Then
	\begin{equation}\label{jen}
		\|P\|_{L^2(\mu)}^2\geq S(K,w) \exp\left[\int \log|P^2| \,d\nu_K\right].
	\end{equation}
\end{theorem}
\begin{proof}
	Since $\nu_K$ is a probability measure, it follows from Jensen's inequality that
	\begin{align}
		\|P\|_{L^2(\mu)}^2&= 
		\exp\left[\log \int\left( \left|P^2\right|w\right) \,d\nu_K\right]
		\nonumber\\
		&\geq
		\exp\left[ \int \log\left(\left|P^2\right|w\right) \,d\nu_K\right]
		\nonumber\\
		&= \exp\left[ \int \log|P^2| \,d\nu_K\right]\exp\left[ \int \log{w} \,d\nu_K\right] \nonumber\\
		&= S(K,w) \exp\left[\int \log|P^2| \,d\nu_K\right].
	\end{align}
\end{proof}

We show that for any non-pluripolar compact $K\subset \mathbb{C}^n$, these integrals $\int \log|P^2| \,d\nu_K$ are finite for holomorphic polynomials $P\not \equiv 0$. To this end, we will need the following version of the Chern--Levine--Nirenberg estimate (cf., Theorem 1.8 of \cite{demailly}). 
\medskip

\begin{proposition}\label{cln} Let $K$ be a compact subset of $\Omega\subset\mathbb{C}^n$. There is a constant $C_{K,\Omega}$ such that for every $
v\in\operatorname{PSH}(\Omega)
$ 
and 
$
u_1,\dots,u_p\in \operatorname{PSH}(\Omega)\cap L^\infty(\Omega),\quad 1\le p\le n,
$ 
one has
\[
\|v\, dd^c u_1\wedge\cdots\wedge dd^c u_p\|_K \le C_{K,\Omega}\|v\|_{L^1(\Omega)}
\|u_1\|_{L^\infty(\Omega)}\cdots\|u_p\|_{L^\infty(\Omega)}.
\]

\end{proposition}

\noindent The norm on the left is the mass norm of the current $v\, dd^c u_1\wedge\cdots\wedge dd^c u_p$ 
on $K$; for $p=n$, this is simply $$\int_K |v|  dd^c u_1\wedge\cdots\wedge dd^c u_n.$$

Thus for $\Omega$ bounded, $K\subset\Omega$ compact, and any $u\in\operatorname{PSH}(\Omega)\cap L^\infty(\Omega)$, taking 
$v=\log|P|$ for a polynomial $P$ and $u_1=\cdots =u_n=u$ gives
\[
\int_K |v|(dd^c u)^n \le C_{K,\Omega}\Biggl[\int_\Omega |v|\,dm\Biggr]
\Bigl(\|u\|_{L^\infty(\Omega)}\Bigr)^n < \infty,
\]
since plurisubharmonic functions are locally integrable (here $dm$ denotes Lebesgue measure). This holds, in particular, if $K$ is non-pluripolar and 
\[
(dd^c u)^n=(dd^c V_K^*)^n=(2\pi)^n d\nu_K
\]
is the Monge--Amp\`ere measure of the extremal function $V_K^*$.

We can use Theorem \ref{threeone} to find lower bounds for $L^2-$norms of polynomials.

\begin{example}
	Let $P_d\in \mathcal P(\alpha(i))$ with degree $d=|\alpha(i)|$ and let $K$ be the Euclidean unit ball in $\bb C^2$. Let $d\sigma$ be normalized surface area measure on $\partial K$. 
	By \cite[Lemma 10.13]{demailly}, we have 
\begin{equation}\label{o3}
		\int \log{|P_d|} \, d\sigma \geq \log{\|P_d \|_K}- d \inf_{0<r<1}\left(\frac{(1+r)^{2n-1}}{1-r}. \log{\frac{1}{r}} \right)
\end{equation}
By Example \ref{ex2} above, $\tau^-(K)= \sqrt{2}/{2}$. Hence, it follows from Theorem \ref{tau} that 
	\begin{equation}\label{o2}
		\|P_d\|_K\geq \left(\frac{\sqrt{2}}{2} \right)^d.
	\end{equation}
Substituting \eqref{o2} into \eqref{o3}, we get the following:
\begin{equation}\label{o1}
	\int \log{|P_d|} \, d\sigma \geq d \left[\log{\frac{\sqrt{2}}{2}}-\inf_{0<r<1}\left(\frac{(1+r)^{2n-1}}{1-r} \log{\frac{1}{r}} \right) \right].
\end{equation}
Note that $d\sigma= d\nu_K$, see \cite[Example 6.5.6]{K}. Let $\mu=w d\nu_K$ be a finite Borel measure with $S(K,w)>0$ and $P_d$ be a monic polynomial of degree $d$ as above. Then in view of \eqref{o1} and \eqref{jen}, we get
\begin{equation}\label{seq1}
	\|P_d\|_{L^2(\mu)}^2\geq \frac{S(K,w)}{2^{d} \exp\left[ 2d \inf_{0<r<1}\left(\frac{(1+r)^{2n-1}}{1-r} \log{\frac{1}{r}}\right)\right]}.  
\end{equation} 
\end{example}

The inequality \eqref{mah1} below is due to Mahler \cite{mahler}.
\begin{theorem}\label{mah the}
	Let $K= \overline{\mathbb{D}}^n$ be the unit polydisk and $P\in \mathcal P(\alpha(i))$ be a non-constant polynomial where $\alpha(i)=(\alpha_1,\ldots, \alpha_n)$. Let $d\nu =\frac{d\theta_1\cdots d\theta_n}{(2\pi)^n}$ where $\frac{d\theta_j}{2\pi}$ is the normalized arclength measure on the $j$-th copy of $\partial \mathbb{D}$ . Then 
	\begin{equation}\label{mah1}
		\exp\left[\int \log{|P|}\, d\nu \right] \geq \frac{1}{\binom{|\alpha(i)|}{\alpha_1} \cdots \binom{|\alpha(i)|}{\alpha_n}}
	\end{equation}

\end{theorem}
\begin{proof}
The inequality \eqref{mah1} follows from  \cite[eq. (3)]{mahler} where $1$ in the numerator of the right side of the inequality \eqref{mah1} is just the leading coefficient of $P$.
\end{proof}

\begin{remark}
 Note that the normalized arclength measure is indeed the equilibrium measure of $ \overline{\mathbb{D}}$. Hence, for   $K= \overline{\mathbb{D}}^n$, we have $d\nu=\frac{d\theta_1\cdots d\theta_n}{(2\pi)^n}=d\nu_K$, see  \cite[Theorem 1]{bloc}. Let $\mu=w d\nu_K$ be a finite Borel measure with $S(K,w)>0$ Then, in view of \eqref{mah1} and \eqref{jen} we get
\begin{equation}\label{mah2}
\|P\|_{L^2(\mu)}^2\geq \frac{S(K,w)}{\left[\binom{|\alpha(i)|}{\alpha_1} \cdots \binom{|\alpha(i)|}{\alpha_n}\right]^2}.
\end{equation}	
where $P$ is as in Theorem \ref{mah the}.
\end{remark}

\section{Orthogonal and Chebyshev Polynomials on Product Sets}

We study orthogonal and Chebyshev polynomials on product sets 
$K = K_1 \times \cdots \times K_n$, where each $K_j$ is a non-polar compact set in $\bb C$. We start with a basic observation.
\begin{theorem}
	Let $K=K_1\times \cdots \times K_n$, where each $K_j$, $j=1,\ldots, n$, is a non-polar compact subset of $\mathbb{C}$. Then 
	\begin{equation}\label{minus}
		\tau^-(K)=\min_{j} \ca(K_j).
	\end{equation}
\end{theorem}
\begin{proof}
	We have (see \cite[Section 7, Property (e')]{Zak1}), $c(K)=\min  \ca(K_j)$ and in view of \eqref{various cap}
	\begin{equation}\label{c11}
		\tau^-(K)\geq c(K).
	\end{equation}
	Let $j$ be such that $\ca(K_j)=c(K)$ and let $T^{(K_j)}_k$ be the $k$-th unweighted Chebyshev polynomial for $K_j$. 
	Since $$\left\|T^{(K_j)}_k\right\|_K =\left\|T^{(K_j)}_k\right\|_{K_j}$$ and from \cite[Corollary 5.5.5]{Ran95} $$\left\|T^{(K_j)}_k\right\|_{K_j}^{1/k} \rightarrow \ca(K_j),$$ it follows from the definition of $\tau^-(K)$ that 
	\begin{equation}\label{c12}
		\ca(K_j)\geq \tau^-(K).
	\end{equation}
	Combining \eqref{c11} and \eqref{c12} gives the result.
\end{proof}

We consider weight functions of the form:
\[
w(z_1, \dots, z_n) \;=\; w_1(z_1)\,\cdots\, w_n(z_n).
\]

First, we consider orthogonal polynomials. Let $\mu_{K_j}$ denote the equilibrium measure of $K_j$. Then the Monge--Amp\`ere measure $\nu_K$ takes the following form (see  \cite[Theorem 1]{bloc}):
\begin{equation}
	\nu_K=\mu_{K_1} \otimes \cdots \otimes \mu_{K_n}.
\end{equation}
We consider  
\[
d\mu_j(z_j) \;=\; w_j(z_j)\, d\mu_{K_j}(z_j),
\]
where each $\mu_j$ is a finite positive Borel measure on $K_j$. Let
\[
\mu \;=\; \mu_1 \otimes \cdots \otimes \mu_n
\]
We always assume the Szeg\H{o} condition
\[
S(K_j,w_j) \;=\; \exp \left[\int \log w_j \,d\mu_{K_j} \right]> 0
\]
holds for each $j$.

Then, in view of Szeg\H{o}'s condition, $E_j:=\mathrm{supp}(\mu_j)= \mathrm{supp}(\mu_{K_j})$ and $\mathrm{Cap}(E_j)=\mathrm{Cap}(K_j)$, see Lemma~1.2.7 in \cite{ST92}. Hence,
the monomials are linearly independent in $L^2(\mu)$ 
(see p.\ 435--436 in \cite{Bloom} or Example~4.4 in \cite{Siciak}).
An important observation we use below is that as a consequence of Jensen's inequality we have
\begin{equation}\label{additional lower}
\int\,d\mu_j=\exp\left[\log\int w_j\, d\mu_{K_j}\right]\geq \exp\left[\int \log{w_j}\, d\mu_{K_j}\right]=S(K_j,w_j).
\end{equation}

\medskip

Let $V_{\alpha(i)}\in \mathcal P(\alpha(i))$ be an orthogonal polynomial in $L^2(\mu)$. Thus its leading term is
\[
z^{\alpha(i)} \;=\; z_1^{\alpha_1}\,\cdots\, z_n^{\alpha_n}.
\]
Let $P_{\alpha_j}^{(j)} $ be a monic orthogonal polynomial with respect to $\mu_j$ where we assume $P_0^{(j)}=1$.  
We will show that
\begin{equation}\label{orthopol}
V_{\alpha(i)}(z_1,\ldots, z_n) \;=\; \prod_{j=1}^n P_{\alpha_j}^{(j)}(z_j).
\end{equation}
It is enough to show that $V_\alpha\,\perp \,z^\beta$ in $L^2(\mu)$ if $\beta \prec \alpha$ where $\beta=(\beta_1,\ldots, \beta_n)$.  
Note that $\beta \prec \alpha$ implies there is at least one index $k$ such that $\beta_k < \alpha_k$.

\bigskip

\noindent
\subsection{Fubini--Tonelli Argument and Factorization}

Let $\beta\prec \alpha$. Note that $\lvert V_{\alpha(i)}\,z^\beta \rvert$ is integrable with respect to $\mu$, so the 
Fubini--Tonelli theorem is applicable to obtain
\begin{align}
	&\int V_{\alpha(i)}\,\overline{z^\beta} \,d\mu\\
	 &=
	\int(P_{\alpha_k}^{(k)}(z_k) \overline{z_k^{\beta_k}}\, d\mu_{k}(z_k) )\\
	 &\times\int \cdots \int \prod_{\substack{{j=1}\\ j\neq k}}^n  P_{\alpha_j}^{(j)}(z_j) \overline{z_j^{\beta_j}} \, d\mu_1(z_1)\cdots d\mu_{k-1}(z_{k-1})\, d\mu_{k+1}(z_{k+1})\cdots d\mu_n(z_n) \\
	&=0.
\end{align}
This verifies (\ref{orthopol}).  
\medskip

Now let us show that
\begin{equation}\label{szeg}
S(K,w) 
\;=\; 
\exp\left(\int \log w \, d\nu_K\right)
\;=\;
\prod_{j=1}^n S(K_j, w_j).
\end{equation}
We have
\[
\log w(z_1,\dots,z_n)
\;=\;
\log\bigl(w_1(z_1)\bigr) \;+\; \cdots \;+\; \log\bigl(w_n(z_n)\bigr)
\;=\;
\sum_{j=1}^n \log w_j(z_j).
\]
Since each $\log w_j \in L^1(\mu_{K_j})$, and $|\log{w}(z_1,\ldots, z_n)|\leq \sum_{j=1}^n|\log\bigl(w_n(z_n)|$, by Fubini--Tonelli we get
\[
\int \log w \;d\nu_K 
\;=\; 
\int \sum_{j=1}^n \log w_j\, d\nu_K
\;=\; 
\sum_{j=1}^n \int \log w_j \, d\mu_{K_j},
\]
which proves (\ref{szeg}).

\medskip

Since $V_{\alpha(i)} = \prod_{j=1}^n P_{\alpha_j}^{(j)}$, another application of Fubini--Tonelli shows
\begin{align}
	\biggl(\int \lvert V_{\alpha(i)}\rvert^2 \, d\mu\biggr)^{\!\!1/2}
	=
	\biggl(\int \prod_{j=1}^n \lvert P_{\alpha_j}^{(j)}\rvert^2 \,d\mu \biggr)^{\!\!1/2}
	&=
	\prod_{j=1}^n 
	\biggl(\int \lvert P_{\alpha_j}^{(j)}\rvert^2 \,d\mu_j\biggr)^{\!\!1/2}\\
	&= \prod_{j=1}^n \|P_{\alpha_j}^{(j)}\|_{L^2(\mu_j)}.
\end{align}

\bigskip

\noindent
\subsection{Universal lower bounds and Consequences}
In view of the existence of universal lower bounds for the $L^2$ norms for the $\alpha_j>0$ case (see e.g.\cite{Alp19}, \cite{AlpZin20}), \eqref{additional lower} for the $\alpha_j=0$ case and \eqref{minus}, 
it follows that
\begin{align}
\|V_{\alpha(i)}\|_{L^2(\mu)}^2\label{prodeq}
\;=\;
\prod_{j=1}^n \bigl\|P_{\alpha_j}^{(j)}\bigr\|_{L^2(\mu_j)}^2
\;&\geq\;
\prod_{j=1}^n S(K_j,w_j) \prod_{j=1}^n [\ca(K_j)^{\alpha_j}]^2\\
\;&=\;
S(K,w)\prod_{j=1}^n [\ca(K_j)^{2\alpha_j}]\\
&\geq S(K,w) \tau^{-}(K)^{2|\alpha(i)|}.\label{sup11}
\end{align}

Hence 
\begin{equation}\label{unif l2}
	\bigl[W_{2,\alpha(i)}(K,w)\bigr]^2
	\;\ge\;
	S(K,w).
\end{equation}
This result for product sets gives a generalization of the one-dimensional result { (1.8)}.

\medskip

For a compact subset $L$ of $\bb C$, let $\Omega_L$ denote the unbounded component of $\overline{\bb C}\setminus L$. From the above discussion, we can derive the following results for the Monge-Amp\`ere measure $\nu_K$ of a product set in both $\bb C^n$ and $\bb R^n$ settings.
\begin{theorem}\label{fourthree}
	Let $K=K_1\times \cdots \times K_n$, where each $K_j$ is a regular compact subset of $\mathbb{C}$. Let $V_{\alpha(i)}\in \mathcal P(\alpha(i))$ be an orthogonal polynomial for $\nu_K$ with leading term $z^{\alpha(i)}= z_1^{\alpha_1}\cdots z_n^{\alpha_n}$. Then
	\begin{equation}\label{prod eq1}
		\|V_{\alpha(i)}\|_{L^2(\nu_K)}\geq  \prod_{j=1}^n \ca(K_j)^{\alpha_j}
	\end{equation}
	and thus 
	
	\begin{equation}\label{lower widom l2}
	\bigl[W_{2,\alpha(i)}(K,1)\bigr]^2
	\;\ge\;
	1.
	\end{equation}
	Equality in \eqref{prod eq1} holds for some $\alpha(i)=(\alpha_1,\ldots, \alpha_n)$ if and only if 
	for each $j\in\{1,\ldots, n\}$ one of the following conditions is satisfied:
	\begin{enumerate}
		\item $\alpha_j=0$, or
		\item $\partial \Omega_{K_j} = R_{\alpha_j}^{-1}(\partial \bb D) $ for some polynomial $R_{\alpha_j}$ of degree $\alpha_j$.
	\end{enumerate}
	In this case, for any $c\in \bb N$, 
	\begin{equation}\label{prod eq3}
		\|V_{c\alpha(i)}\|_{L^2(\nu_K)}= \prod_{j=1}^n \ca(K_j)^{c\alpha_j}.
	\end{equation}

\end{theorem}
\begin{proof}
	The inequality \eqref{prod eq1} was already proved in \eqref{prodeq} and \eqref{lower widom l2} is a special case of \eqref{unif l2}. It follows from the universal lower bounds and \eqref{prodeq} that equality in \eqref{prod eq1} is satisfied for some $\alpha(i)$ if and only if for each $j$ we have
	\begin{equation}\label{nonprod}
		 \bigl\|P_{\alpha_j}^{(j)}\bigr\|_{L^2(\mu_{K_j})} = \ca(K_j)^{\alpha_j}.
	\end{equation}	
	It was shown in \cite[Theorem 4.2]{AlpZin21} that for the case $\alpha_j\geq 1$, \eqref{nonprod} holds if and only if $\partial \Omega_{K_j} = R_{\alpha_j}^{-1}(\partial \bb D) $ for some polynomial of degree $\alpha_j$. For the case $\alpha_j=0$, \eqref{nonprod} is automatically satisfied. Thus equality holds in \eqref{prod eq1} for some $\alpha(i)$ if and only if one of the given conditions in the statement of the theorem is satisfied for each $j$. In this case, \eqref{prod eq3} holds, since $\|P_{c\alpha_j}\|_{L^2(\mu_{K_j})}= \ca(K_j)^{c\alpha_j}$ holds in view of \cite[Theorem 4.2]{AlpZin21}.
\end{proof}
\begin{theorem}
	Let $K=K_1\times \cdots \times K_n$ where each $K_j$, $j=1,\ldots, n$, is a regular compact subset of $\mathbb{R}$. Let $V_{\alpha(i)}\in \mathcal P(\alpha(i))$ be an orthogonal polynomial for $\nu_K$ with the leading term $z^{\alpha(i)}= z_1^{\alpha_1}\cdots z_n^{\alpha_n}$. For a given $\alpha(i)=(\alpha_1,\ldots, \alpha_n)$, let $c_j=0$ if $\alpha_j=0$, and $c_j=1$ if $\alpha_j>0$. Then
	\begin{equation}\label{prod eq11}
		\|V_{\alpha(i)}\|_{L^2(\nu_K)}\geq  \prod_{j=1}^n \left(\sqrt{2}\right)^{c_j}\prod_{j=1}^n \ca(K_j)^{\alpha_j},
	\end{equation}
	and, in particular, 
	\begin{equation}\label{prod eq12}
	[W_{2,\alpha(i)}(K,1)]^2\geq 2, \mbox{ for } \, i\geq 1.
	\end{equation}
		Equality in \eqref{prod eq11} holds for some $\alpha(i)=(\alpha_1,\ldots, \alpha_n)$ if and only if 
	for each $j\in\{1,\ldots, n\}$ one of the following  two conditions is satisfied:
	\begin{enumerate}
		\item $\alpha_j=0$, or
		\item ${K_j} = R_{\alpha_j}^{-1}([-1,1])$ for some polynomial $R_{\alpha_j}$ of degree $\alpha_j$.
	\end{enumerate}
\end{theorem}
\begin{proof}
For $\alpha_j>0$, we have (see \cite[Theorem 3.1.]{AlpZin20})
	\begin{equation}\label{ooo}
		\|P_{\alpha_j}^{(j)}\|_{L^2(\mu_{K_j})}^2 \geq 2 \ca(K_j)^{2\alpha_j}. 
	\end{equation}
	Thus, \eqref{prod eq11} and \eqref{prod eq12} hold in view of \eqref{minus} and the equality part in \eqref{prodeq}.
	Equality in \eqref{ooo} holds if and only if ${K_j} = R_{\alpha_j}^{-1}([-1,1])$ for some polynomial $R_{\alpha_j}$ of degree $\alpha_j$, see \cite[Theorem 4.4]{AlpZin21}. Thus equality in \eqref{prod eq11} holds if and only if for each $\alpha_j$ one of the two conditions described in the statement of the theorem holds.
\end{proof}

\noindent
\subsection{The Sup-norm Case}

\medskip

We continue with $K$ and $w$ in this product setting in $\bb C^n$. We assume each $w_j$ is nonnegative and Borel measurable with $\sup_{{z_j}\in K_j}\,|w_j(z_j)| < \infty$. Let
\begin{equation}
	\widehat{w_j}(z_j):= \lim_{r\rightarrow 0^+} \sup_{\xi\in B_r(z_j)\cap K_j}w_j(\xi),\,\, z_j\in K_j
\end{equation}
so that each $\widehat{w_j}$ is upper semicontinuous, bounded, and nonnegative on $K$ with 
\begin{equation}
	w_j\leq \widehat{w_j}
\end{equation}
and
\begin{equation}\label{s comp}
	0<S(K_j,w_j)\leq S(K_j, \widehat{w_j}).
\end{equation}
Then it is easy to see that 
\begin{equation}
	\widehat{w}(z_1,\ldots, z_n)= \prod_{j=1}^n \widehat{w_j}(z_j).
\end{equation}

We next prove an analogue of \eqref{unif l2} for the sup-norm case, which is a generalized version of \cite[Theorem 13]{NSZ21}. 
We let
\[
d\mu_j(z_j) \;=\; \widehat{w_j}^2(z_j)\, d\mu_{K_j}(z_j),
\]
and
\[
\mu \;=\; \mu_1 \otimes \cdots \otimes \mu_n=\widehat{w}^2\,d\nu_K.
\]
Note that here the Radon-Nikodym derivatives are slightly different compared to the previous subsection.

 Since $d\nu_K$ is a unit measure, Hölder's inequality gives
\begin{equation}\label{holder comp}
\int {\left|T_{\alpha(i),\widehat{w}}^{(K)}\right|}^2 (\widehat{w})^2 d\nu_K \leq	\left\|\left(T_{\alpha(i),\widehat{w}}^{(K)}\widehat{w}\right)^2\right\|_K .
\end{equation}
Let $P_{\alpha(i)}\in \mathcal P(\alpha(i))$ be an orthogonal polynomial with leading term $z^{\alpha(i)}$ for $\mu$. Using the minimality of the $L^2(\mu)$ norm of $P_{\alpha(i)}$ together with  \eqref{holder comp} and \eqref{sup11}, we get
\begin{align}
\left\|(T_{\alpha(i),\widehat{w}}^{(K)}\widehat{w})^2\right\|_K &\geq	\int {\left|T_{\alpha(i),\widehat{w}}^{(K)}\right|}^2 (\widehat{w})^2 d\nu_K \\
&\geq 	\int {\left|P_{\alpha(i)}\right|}^2 (\widehat{w})^2 d\nu_K\\
& \geq S(K,\widehat{w}^2) \tau^{-}(K)^{2|\alpha(i)|}\\
&= \left[S(K,\widehat{w})\right]^2 \tau^{-}(K)^{2|\alpha(i)|}.
\end{align}
Thus, { generalizing (1.9), we have}
\begin{equation}\label{hat wid}
	W_{\infty,\alpha(i)}(K,\widehat{w})\geq S(K,\widehat{w}).
\end{equation}


\subsection{Further Results on Chebyshev Polynomials in the Real Case}

\medskip
We study the special case $K_j\subset \bb R$ for each $j$. We use the same notation from the previous subsection. The main result in this section is Theorem \ref{big prod} and our proof relies heavily on arguments from \cite{reimer}. 

It was proved in \cite[Corollary 6 and Theorem 7]{NSZ21} that, for $\alpha_j>0$, the $\alpha_j$-th Chebyshev polynomial $T_{\alpha_j, \widehat {w_j}}^{(K_j)}$ for $\widehat {w_j}$ on $K_j$ is unique, its zeros and coefficients are real, and there are points $x_0^{(j)}, \ldots, x_{\alpha_j}^{(j)}\in K_j$ such that
\begin{equation}
	x_0^{(j)}>x_1^{(j)}>\cdots > x_{\alpha_j}^{(j)} 
\end{equation} 
which satisfy
\begin{equation}\label{alternation}
	\widehat{w_j}(x_k^{(j)})T_{\alpha_j, \widehat {w_j}}^{(K_j)}(x_k^{(j)})= (-1)^k \|T_{\alpha_j, \widehat{w_j}}\widehat{w_j}\|_{K_j}, \,\,\, k=0,\ldots, \alpha_j.
\end{equation}
Thus we have a generalized version of Chebyshev's alternation theorem and in particular the zeros of these Chebyshev polynomials are simple. Let $L_j:=\{x_0^{(j)}, \ldots, x_{\alpha_j}^{(j)}\}$.
For the case $\alpha_j=0$, by compactness of $K_j$ and upper semicontinuity of $\widehat{w_j}$, there is a point $x_0^{(j)}\in K_j$ such that $\widehat{w_j}\left(x_0^{(j)}\right)= \|\widehat{w_j}\|_{K_j}$. Thus we let $L_j:=\{x_0^{(j)}\}$ for such a point $x_0^{(j)}$ in this case.

\medskip

Let $L = L_1 \times \cdots \times L_n$ be a cartesian product of these real nodes. Thus $L$ consists of $\prod_{j=1}^n (\alpha_j+1)$ points on $K$. Since we consider a product of subsets of the real line, we use $x^\beta=x_1^{\beta_1}\cdots x_n^{\beta_n}$ to denote monomials instead of $z^\beta$.

Let 
\begin{equation}\label{prod def}
	Q_{\alpha(i)}(x_1,\ldots,x_n)=\prod_{j=1}^n T_{\alpha_j, \widehat {w_j}}^{(K_j)}(x_j).
\end{equation}
Then $Q_{\alpha(i)}\in \mathcal P(\alpha(i))$; i.e., its leading term is $x^{\alpha(i)}=x_1^{\alpha_1}\cdots x_n^{\alpha_n}$; and it has real coefficients. In view of \eqref{prod def} and \eqref{alternation},
for any $x=(x_{k_1}^{(1)},\ldots, x_{k_n}^{(n)})\in L$, we have
\begin{equation}
Q_{\alpha(i)}(x) \widehat{w}(x)= (-1)^{k_1+\ldots+k_n} \|Q_\alpha \widehat{w}\|_K.
\end{equation}
The goal of this section is to prove that $Q_{\alpha(i)}$ is a weighted Chebyshev polynomial for $\widehat{w}$ on $K$. 

Let 
\[
\bb P_{\alpha(i)}^{(n)} \;:=\; \mathrm{span}\bigl\{\,x^\beta : |\beta|\le |\alpha(i)| \bigr\},
\]
where the coefficients in the linear combinations are taken to be real. Let $\widetilde{\bb P_{\alpha(i)}^{(n)}}$ be the subspace of $\bb P_{\alpha(i)}^{(n)}$ where the coefficient of $x^{\alpha(i)}$ is $0$. { Note that if $p,q\in \mathcal P(\alpha(i))$ then $p-q\in \widetilde{\bb P_{\alpha(i)}^{(n)}}$. }Also, note that $T_{\alpha(i), \widehat{w}}^{(K)}$ may be assumed to have real coefficients, since 
\[
\left\|\mathrm{Re}\,T_{\alpha(i), \widehat{w}}^{(K)} \widehat{w} \right\|_{K}
\;\le\;
\left\|T_{\alpha(i), \widehat{w}}^{(K)} \widehat{w} \right\|_{K}.
\]
Thus we only consider real-coefficient polynomials in this subsection.

\medskip

\noindent
\begin{lemma}\label{extre}
	If a real polynomial $R_{\alpha(i)}\in \mathcal P(\alpha(i))$, i.e., with leading term $x^{\alpha(i)}$, satisfies:
	\begin{itemize}
	\item[(i)] $|R_{\alpha(i)}(\xi) \widehat{w}(\xi)| = \|R_{\alpha(i)}\widehat{w}\|_K$ for all $\xi \in L$; and 
	\item[(ii)] there are $d(\xi)>0$, $\xi\in L$ such that $\displaystyle \sum_{\xi\in L}d(\xi) R_{\alpha(i)}(\xi) G(\xi) = 0$  for all $G\in \widetilde{\bb P_{\alpha(i)}^{(n)}}$,
	\end{itemize}
	then $R_{\alpha(i)}$ is a Chebyshev polynomial 
	for $\widehat{w}$ on $K$.
\end{lemma}

\medskip

\noindent
\begin{proof}
	Let $R_{\alpha(i)}\in \mathcal P(\alpha(i))$ be such a polynomial. Assume there is a $P\in \mathcal P(\alpha(i))$ such that
\begin{equation}\label{not imp}
	\|P\widehat{w}\|_{K}<	\|R_{\alpha(i)} \widehat{w}\|_{K}.
\end{equation}
We write $P = R_{\alpha(i)}-U$ where $U\in \widetilde{\bb P_{\alpha(i)}^{(n)}}$.
	In view of (ii), there is a $\xi_0 \in L$ such that $R_{\alpha(i)}(\xi_0) (R_{\alpha(i)}(\xi_0)-P(\xi_0) )\leq 0$. This implies that $|P(\xi_0)|\geq |R_{\alpha(i)}(\xi_0)|$. This contradicts \eqref{not imp}, since by $(i)$ we have $|R_{\alpha(i)}(\xi_0) \widehat{w}(\xi_0)| = \|R_{\alpha(i)}\widehat{w}\|_K$. This proves that such a polynomial $R_{\alpha(i)}$ is a weighted Chebyshev polynomial for $\widehat{w}$ on $K$.
	
	\medskip

\end{proof}

\begin{theorem}\label{big prod}
	Let $K = K_1 \times \cdots \times K_n$, where each $K_j \subset \mathbb{R}$ is non-polar, and let 
	\[
	w(x_1,\ldots, x_n) = w_1(x_1)\,\dots\,w_n(x_n)
	\]
	be such that each $w_j$ is bounded and Borel measurable with $S(K_j, w_j)>0$. Then
\begin{equation}
	Q_{\alpha(i)}(x_1,\ldots, x_n):= \prod_{j=1}^n T_{\alpha_j,\widehat{w_j}}^{(K_j)}(x_j)
\end{equation}
	is a Chebyshev polynomial for $\widehat{w}$ on $K$ with leading term 
	\(
	x_1^{\alpha_1}\cdots x_n^{\alpha_n}.
	\)
\end{theorem}
\begin{proof}
For any real-valued function $f_j$ defined on $L_j $, 
the \emph{divided difference} of $f_j$ on nodes 
\[
x_0^{(j)} , x_1^{(j)} , \cdots , x_{\alpha_j}^{(j)}
\]
with $\alpha_j\not = 0$ is given by
\begin{equation}\label{divid 1}
	f_j[x_0^{(j)},\,\dots,\,x_{\alpha_j}^{(j)}]
	\;=\;
	\sum_{i=0}^{\alpha_j}
	\; \frac{f_j\bigl(x_i^{(j)}\bigr)}{v_j^\prime\bigl(x_i^{(j)}\bigr)},
\end{equation}
where $v_j\left(x^{(j)}\right)=\prod_{i=0}^{\alpha_j} \left(x^{(j)} - x_{i}^{(j)}\right)$. If $\alpha_j=0$, then we set
\begin{equation}\label{divid 2}
v_j\left(x^{(j)}\right)=\left(x^{(j)}-x_0^{(j)}\right)\,\,\, \mbox{ and call }	f_j\left[x_0^{(j)}\right]= f_j\left(x_0^{(j)}\right)\ \
\end{equation}
the divided difference on $L_j$ of $f_j$. When $f_j$ is a polynomial of degree at most $\alpha_j$, 
the coefficient of $x_j^{\alpha_j}$ in $f_j$ is given by this divided difference.

Let $F \in \bb P_{\alpha(i)}^{(n)}$. Following Reimer \cite{reimer}, we define the \emph{multivariate} divided difference $[F; L]$ of $F$ on $L$ by applying 
the univariate divided-difference operator for each $j=1,\dots,n$ on the appropriate $x_j$-nodes belonging to $L_j$.  
One checks that for a pure monomial $x^\beta$, the order of application to different coordinates does not matter, 
so that $[x^\beta; L]$ is independent of the sequence in which the differences are taken.  
Hence, 
\[
\left[x^{\alpha(i)}; L\right] \;=\; 1, \ \hbox{and} \quad
[P; L] \;=\; 0 \;\text{if } P\in \widetilde{\bb P_{\alpha(i)}^{(n)}}.
\]
For $\xi=\left(x_{k_1}^{(1)},\ldots, x_{k_n}^{(n)}\right)\in L$, let
\begin{equation}\label{c def}
	c(\xi)=\prod_{j=1}^n \frac{1}{v_j^\prime\left(x_{k_j}^{(j)}\right)}.
\end{equation} 
Using linearity of divided difference operators, in view of \eqref{divid 1}, \eqref{divid 2} and \eqref{c def}, we obtain
\begin{equation}\label{f divid}
	[F;L]=\sum_{\xi\in L} c(\xi) F(\xi).
\end{equation}
We see that $v_j^\prime$ and $T_{\alpha_j,\widehat{w_j}}^{(K_j)}$ have the same sign at all points in $L_j$. In view of \eqref{divid 1},
 \eqref{divid 2} and \eqref{f divid}, this leads to
\begin{equation}
	[Q_{\alpha(i)};L]= \sum_{\xi\in L} c(\xi) Q_{\alpha(i)}(\xi)
\end{equation}
where
\begin{equation}
	c(\xi) Q_{\alpha(i)}(\xi)>0 \,\,\, \mbox{for all} \,\,\, \xi\in L.
\end{equation}
Note that 
\begin{equation}
	| Q_{\alpha(i)}(\xi)\widehat{w}(\xi)|=\|Q_{\alpha(i)} \widehat{w}\|_K \mbox{  for all} \,\,\, \xi\in L,
\end{equation}
thus condition $(i)$ in Lemma \ref{extre} holds for $R_{\alpha(i)}=Q_{\alpha(i)}$.

For each $\xi \in L$, since $\widehat{w}(\xi)>0$,
\begin{equation}\label{c eq}
	c(\xi)=\frac{|c(\xi)|Q_{\alpha(i)}(\xi) \widehat{w}(\xi)}{\|Q_{\alpha(i)} \widehat{w}\|_K}.
\end{equation}
Let
\begin{equation}\label{d eq}
	d(\xi)= \frac{|c(\xi)| \widehat{w}(\xi)}{\|Q_{\alpha(i)} \widehat{w}\|_K},
\end{equation}
so that $d(\xi)>0$ on $L$. Combining, \eqref{f divid}, \eqref{c eq} and \eqref{d eq}, we get
\begin{equation}
	[F;L]= \sum_{\xi\in L} d(\xi) Q_{\alpha(i)}(\xi) F(\xi).
\end{equation}
Since $[F;L]=0$ for $F\in \widetilde{\bb P_{\alpha(i)}^{(n)}}$, this implies that condition $(ii)$ in Lemma \ref{extre} is also satisfied by $R_{\alpha(i)}=Q_{\alpha(i)}$. Therefore, $Q_{\alpha(i)}$ is a Chebyshev polynomial with leading term $x^{\alpha(i)}=x_1^{\alpha_1}\cdots x_n^{\alpha_n}$ with respect to $\widehat{w}$ on $K$.
\end{proof}
We have the following corollary.
\begin{theorem}\label{foursix}
	Let $K=K_1\times \cdots \times K_n$, where each $K_j$, $j=1,\ldots, n$, is a non-polar compact subset of $\mathbb{R}$. For given $\alpha(i)=(\alpha_1,\ldots, \alpha_n)$, let $c_j=0$ if $\alpha_j=0$, and $c_j=1$ if $\alpha_j>0$. Then
	\begin{equation}\label{prod eq13}
		\left\|T_{\alpha(i), 1}^{(K)}\right\|_{K}\geq  \prod_{j=1}^n 2^{c_j}\prod_{j=1}^n \ca(K_j)^{\alpha_j}
	\end{equation}
	and in particular,
	\begin{equation}\label{prod eq14}
	W_{\infty,\alpha(i)}(K,1)\geq 2, \mbox{ for } \, i\geq 1.
	\end{equation}
	Equality in \eqref{prod eq13} holds for some $\alpha(i)=(\alpha_1,\ldots, \alpha_n)$ if and only if 
	for each $j\in\{1,\ldots, n\}$ one of the following two conditions is satisfied:
	\begin{enumerate}
		\item $\alpha_j=0$, or
		\item ${K_j} = R_{\alpha_j}^{-1}([-1,1])$ for some polynomial $R_{\alpha_j}$ of degree $\alpha_j$.
	\end{enumerate}
\end{theorem}
\begin{proof} By Theorem \ref{big prod}, for each $i$, we have
$T_{\alpha(i), 1}^{(K)}(x_1,\ldots,x_n)=\prod_{j=1}^n T_{\alpha_j,1}^{(K_j)}(x_j)$. For $\alpha_j>0$, it follows from \cite{Sch08} that
\begin{equation}\label{prod eq15}
	\left\|T_{\alpha_j, 1}^{(K_j)}\right\|_{K_j}\geq 2 \ca(K_j)^{\alpha_j}.
\end{equation}
For $\alpha_j=0$,
\begin{equation}\label{prod eq16}
	\left\|T_{\alpha_j, 1}^{(K_j)}\right\|_{K_j}= \ca(K_j)^{\alpha_j}
\end{equation}
holds trivially.
Thus, \eqref{prod eq13} holds and \eqref{prod eq14} also holds in view of \eqref{minus}. It was shown in \cite[Theorem 1.1]{CSZ3} (see also \cite{Tot11}) that for $\alpha_j>0$, the equality is satisfied in \eqref{prod eq15} if and only if ${K_j} = R_{\alpha_j}^{-1}([-1,1])$ for some polynomial $R_{\alpha_j}$ of degree $\alpha_j$. Thus equality in \eqref{prod eq13} holds if and only if for each $j$, (1) or (2) is satisfied.
\end{proof}

\section{Generalization of Mahler's theorem}

In this section we will generalize a result of Mahler \cite{mahler} to very general compact sets in $\bb C$ and product sets in $\bb C^n$ for $n>1$. 
For $n\geq 1$, let $K$ be a non-pluripolar compact subset of $\bb{C}^n$ and $P$ be a polynomial. We define the Mahler measure of $P$ relative to $K$ as 
$$ M(P):= \exp \left[\int \log{|P|} \, d\nu_K\right]$$
where we recall that $\nu_K$ is the Monge-Amp\`ere measure of $K$. 
Note if $n=1$ and $\ca(K)=1$ and $K$ is regular, this agrees with the notion given by Pritsker in \cite[p. 61]{igor}. 
Since we will be working with product sets $K=K_1\times \cdots \times K_n$, to clearly distinguish between univariate equilibrium measures and Monge-Amp\`ere measures in $\bb C^n$ for $n>1$, for a non-polar compact subset $E$ of $\bb C$, we let $\mu_E$ denote the equilibrium measure of $E$.

We begin with a univariate result. 
\begin{theorem}
	Let $K$ be a non-polar compact subset of $\bb{C}$ and $$P(z)= a_d z^d+\ldots + a_k z^k+\ldots +a_0.$$
	Then for $k=0,...,d$,
	\begin{equation}\label{lowera}
		|a_k|\leq \frac{{d \choose k} M(P) (\max_{z\in K} |z|)^{d-k}}{\ca(K)^d}.
	\end{equation}
\end{theorem}
\begin{proof} Clearly we can assume $\deg P\geq 1$ so we can take $d\geq 1$. First assume $a_d\neq 0$. Let $c_1, \ldots, c_d$ be the zeros of $P$. By  Vieta's formulas for $k=0,\ldots, d-1$, 
	\begin{equation}\label{vieta}
		|a_k|=\left|\sum a_d c_{j_1}\cdots c_{j_{d-k}}\right|
	\end{equation}
	where the sum is over all possible $d \choose k$ terms $c_{j_1}\cdots c_{j_{d-k}}$. 
	Let $U^{\mu_K}(z)= \int \log{|z-t|}\, d\mu_K(t)$. Then 
	\begin{equation}\label{exp eq 2}
		M(P)= |a_d|\, \exp{\left[\sum_{j=1}^d U^{\mu_K}(c_j)\right]}
	\end{equation}
	and
	\begin{equation}\label{fros low}
		U^{\mu_K}(c_j) \geq \log{\ca(K)}
	\end{equation}
	by Frostman's theorem (see e.g. \cite[Theorem 3.3.4]{Ran95}). Thus for the $k=d$ case, \eqref{lowera} holds in view of \eqref{exp eq 2} and \eqref{fros low}.
	
	Our next goal is to find an upper bound for each $|a_d c_{j_1}\cdots c_{j_{d-k}}|$ in terms of $M(P)$ for the case $k<d$. If at least one $c_{j_l}$ is 0 then trivially
	\begin{equation}
		0=	|a_d c_{j_1}\cdots c_{j_{d-k}}|\leq \frac{ M(P) (\max_{z\in K} |z|)^{d-k}}{\ca(K)^d}.
	\end{equation}
	Assume that each $c_{j_l}$ is non-zero. Then 
	\begin{equation}\label{exp eq 1}
		|a_d c_{j_1}\cdots c_{j_{d-k}}|= |a_d| \exp \left[ {\sum_{l=1}^{d-k} \log{|c_{j_l}|} }\right].
	\end{equation}
	Let $$H(z):= U^{\mu_K}(z)-\log{|z|}.$$  
	It follows from Frostman's theorem that 
	\begin{equation}\label{another h}
		H(z)\geq \log{\ca(K)}- \log{\max_{w\in K}|w }|=: c_K
	\end{equation}
	on $K$. Since $U^{\mu_K}$ is harmonic on $\bb{C} \setminus \mathrm{supp}(\mu_K)$ and $-\log|z|$ is superharmonic on $\bb{C}$, $H$ is superharmonic on $\bb C \setminus \mathrm{supp}(\mu_K)$. From the generalized minimum principle (\cite[Theorem I.2.4]{ST97}), this implies that
	\begin{equation}
		H(z)\geq  c_K 
	\end{equation}
	for $\bb C \setminus \mathrm{supp}(\mu_K)$; combined with \eqref{another h}, this inequality holds for all $z\in \bb C$.
	This implies that 
	\begin{equation}
		\frac{\exp[U^{\mu_K}(z)]}{\exp[\log|z|]}\geq e^{c_K}
	\end{equation}
	and thus
	\begin{equation}\label{q1}
		|z|\leq e^{-c_K} e^{U^{\mu_K}(z)},\,\,\, z\in \bb{C}.
	\end{equation}
	Combining \eqref{q1} and \eqref{exp eq 1}, we get 
	\begin{equation}\label{q5}
		|a_d c_{j_1}\cdots c_{j_{d-k}}|\leq |a_d| e^{-(d-k)c_K}  e^{\sum_{l=1}^{d-k} U^{\mu_K}(c_{j_l})}.
	\end{equation}
	Using \eqref{exp eq 2} and Frostman's theorem, we can write
	\begin{align}
		|a_d|  e^{\sum_{l=1}^{d-k} U^{\mu_K}(c_{j_l})} &\leq \frac{|a_d|e^{\sum_{j=1}^{d} U^{\mu_K}(c_{j})}}{\ca(K)^k} \\
		&=\frac{M(P)}{\ca(K)^k}\label{q2}.
	\end{align}
	Hence \eqref{vieta}, \eqref{q5} and \eqref{q2} yield 
	
	\begin{align}
		|a_k|\ &\leq	\frac{{d\choose k} e^{-(d-k)c_K} M(P)}{\ca(K)^k}\\
		&= \frac{{d\choose k} M(P) (\max_{z\in K}|z|)^{d-k}}{\ca(K)^d}.\label{q3}
	\end{align}
	
	Finally, if $a_d= 0$, then $\deg{P}=d_1<d$. Since (see \cite[Corollary 5.2.2]{Ran95}) \\ 
	\begin{equation}
		\ca\left(\left\{z: |z|\leq \max_{w\in K} |w|\right\}\right)= \max_{z\in K} |z|, 
	\end{equation}
	and $K\subseteq \left\{z: |z|\leq \max_{w\in K} |w|\right\}$, we have  $\ca(K)\leq \max_{z\in K} |z|$. Hence we can rewrite \eqref{q3} as
	\begin{align}
		|a_k|&\leq \frac{{d_1\choose k} M(P) (\max_{z\in K}|z|)^{d_1-k}}{\ca(K)^{d_1}}\\
		&\leq\frac{{d\choose k} M(P) (\max_{z\in K}|z|)^{d-k}}{\ca(K)^d}
	\end{align}
	which completes the proof.
\end{proof}

{ We now state and prove a multivariate result.  

\begin{theorem} Let $K=K_1\times \cdots \times K_n$ where each $K_j$, $j=1,\ldots, n$, is a compact non-polar subset of $\bb{C}$ and 
\begin{equation}\label{form}
	P(z_1,\ldots, z_n) = \sum_{k_1=0}^{m_1} \ldots \sum_{k_n=0}^{m_n} a_{k_1\ldots k_n} z_1^{k_1} \cdots z_n^{k_n}
\end{equation}
where the degree of $P$ in $z_j$ is $m_j$ in the sense that there is at least one non-zero term including $z_j^{m_j}$ where $m_j=0$ is also allowed. 
Then
\begin{equation}\label{K4}
	|a_{k_1\ldots k_n}| \leq M(P) \prod_{N=1}^n \frac{{m_N \choose k_N} (\max_{z_N\in K_N} |z_N|)^{m_N-k_N}} {\ca(K_N)^{m_N}}.
\end{equation}
\end{theorem}

\begin{remark}
	Both \eqref{lowera} and \eqref{K4} coincide with Mahler's results in \cite{mahler} when $K=\partial \bb{D}$ and $K=(\partial \bb{D})^n$ respectively.
\end{remark}

\begin{proof} 

Following Mahler \cite{mahler}, we define $P_{k_1\ldots k_N}$ recursively as follows:
\begin{align}
	P(z_1,\ldots, z_n)&= \sum_{{k_1}=0}^{m_1} P_{k_1}(z_2,\ldots, z_n)z_1^{k_1} \label{P1}\\
	P_{k_1\ldots k_{N-1}}(z_N,\ldots, z_n)&= \sum_{k_N=0}^{m_N} P_{k_1\ldots k_{N}}(z_{N+1},\ldots, z_n) {z_N}^{k_N},\,\,\, N>1	 \label{P2}\\
	P_{k_1\ldots k_n}&= a_{k_1\ldots k_n} \label{P3}
\end{align}
We have 
\begin{equation}\label{M1}
	M(P) =\exp \left[\int \log{|P|} \, d\nu_K\right]= \exp \left[\int \log{|P|} \, d\mu_{K_1}\cdots d\mu_{K_n} \right],
\end{equation}
and we define
\begin{equation}\label{M2}
	M(P_{k_1\ldots k_{N-1}}) := \exp \left[\int \log{|P_{k_1\ldots k_{N-1}}|} \, d\mu_{K_N}\cdots d\mu_{K_n} \right],\,\,\, 1<N\leq n,
\end{equation}
and
\begin{equation}
	M(P_{k_1 \ldots k_n}):= |a_{k_1\ldots k_n}|.
\end{equation}
In view of \eqref{lowera} and \eqref{P1} we see that
\begin{equation}\label{K1}
	|P_{k_1}(z_2,\ldots, z_n)| \leq \frac{{m_1 \choose k_1} (\max_{z_1\in K_1} |z_1|)^{m_1-k_1}\exp{[\int \log|P(z_1,\ldots, z_n)|\, d\mu_{K_1}(z_1)]}}{\ca(K_1)^{m_1}}.
\end{equation}
By taking logarithms, integrating both sides of \eqref{K1} with respect to $d\mu_{K_2}\cdots d\mu_{K_n}$, and then exponentiating both sides, we get

\begin{equation}\label{K2}
	M(P_{k_1}) \leq \frac{{m_1 \choose k_1} (\max_{z_1\in K_1} |z_1|)^{m_1-k_1} M(P)} {\ca(K_1)^{m_1}}.
\end{equation}
Arguing similarly, for each $N$, we get

\begin{equation}\label{K3}
	M(P_{k_1 \ldots k_N}) \leq \frac{{m_N \choose k_N} (\max_{z_N\in K_N} |z_N|)^{m_N-k_N} M(P_{k_1 \ldots k_{N-1}})} {\ca(K_N)^{m_N}}.
\end{equation}
Combining \eqref{P3}, \eqref{K2} and \eqref{K3}, we obtain \eqref{K4}.

\end{proof}}

The case where all the coefficients of a polynomial $P$ are integers is important, especially in terms of applications in number theory. If $K\subset \mathbb{C}$ has $\ca(K)=1$ and $P$ is a non-zero polynomial with integer coefficients  
it follows from \eqref{exp eq 2} and \eqref{fros low} that $M(P)\geq 1$. We refer the reader to the survey of Symth \cite{Symth} for more information and applications in the univariate setting. We now generalize this result as follows: 

\begin{theorem}\label{fivethree}
	Let $P$ be a non-zero polynomial of the form \eqref{form} where all the coefficients are integers. Then
	\begin{equation}\label{m1}
		M(P) \geq \prod_{N=1}^n \min (1, \ca(K_N)^{m_N}).
	\end{equation}
	In particular, if $\ca(K_N) \geq 1$ for $N=1,\ldots, n$ then
	\begin{equation}\label{m2}
		M(P) \geq 1.
	\end{equation}
\end{theorem}
\begin{proof}
	Let $k_1=m_1$.	It follows from \eqref{P1} and \eqref{K2} that
	\begin{equation}\label{s1}
		\ca(K_1)^{k_1} M(P_{k_1}) \leq M(P).
	\end{equation}
	For $N=2$, in \eqref{P2}, let $k_2$ be the largest integer for which $P_{k_1 k_2}$ is a non-zero polynomial. Then if we replace $m_2$ by $k_2$ in \eqref{K3}, it still holds. Hence we have
	\begin{equation}\label{s2}
		\ca(K_2)^{k_2} M(P_{k_1 k_2})  \leq M(P_{k_1}).
	\end{equation}
	For $N=3,\ldots,n$, define $k_N$ recursively as the largest integer for which $P_{k_1\ldots k_N}$ is a non-zero polynomial. Thus we have  
	\begin{equation}\label{s3}
		\ca(K_N)^{k_N}	M(P_{k_1\ldots k_{N}}) \leq M(P_{k_1\ldots k_{N-1}}).
	\end{equation}
	Since $a_{k_1\ldots k_n}$ is a non-zero integer by our assumptions, it follows from \eqref{P3} that
	\begin{equation}\label{s4}
		M(P_{k_1\ldots k_{N}})= |a_{k_1\ldots k_n}|\geq 1.
	\end{equation}
	Combining \eqref{s1}, \eqref{s2}, \eqref{s3}, \eqref{s4}, we obtain
	\begin{equation}\label{s5}
		M(P) \geq \prod_{N=1}^n \ca(K_N)^{k_N}.
	\end{equation}
	Since $0\leq k_N\leq m_N$ for each $N$, the inequality \eqref{s5} implies \eqref{m1}. The inequality \eqref{m2} follows from \eqref{m1}.
\end{proof}


\end{document}